\documentclass{amsart}

\usepackage{latexsym}
\usepackage{amssymb}
\usepackage{amsmath}
\usepackage{amsthm}
\usepackage{hyperref}
\usepackage{xcolor}
\usepackage{indentfirst}

\newcommand{\vvol}{\mathrm{V}}

\newcommand{\sd}{s_{\delta}}
\newcommand{\cd}{c_{\delta}}
\newcommand{\lgra}{\longrightarrow}

\newcommand{\iid}{\mathrm{Id}\,}

\newcommand{\ddiv}{\mathrm{div}}

\newcommand{\trace}{\mathrm{tr\,}}

\newcommand{\Ss}{\mathbb{S}}

\newcommand{\R}{\mathbb{R}}

\newtheorem{example}{Exemples}[section]
\newtheorem{thm}{Theorem}[section]

\newtheorem{lemma}[thm]{Lemma}
\newtheorem{prop}[thm]{Proposition}

\newtheorem{remark}[thm]{Remark}
\newtheorem{remarks}[thm]{Remarks}
\newtheorem{definition}[thm]{Definition}
\newtheorem{notation}[thm]{Notation}
\newtheorem{exabout:ample}[thm]{Example}

\newcommand{\beqt}{\begin{equation}}  \newcommand{\eeqt}{\end{equation}}
\newcommand{\bal}{\begin{align}}      \newcommand{\eal}{\end{align}}
\newcommand{\ba}{\begin{array}}      \newcommand{\ea}{\end{array}}
\newcommand{\bc}{\begin{center}}     \newcommand{\ec}{\end{center}}
\newcommand{\be}{\begin{enumerate}}  \newcommand{\ee}{\end{enumerate}}
\newcommand{\beq}{\begin{eqnarray}}  \newcommand{\eeq}{\end{eqnarray}}
\newcommand{\beQ}{\begin{eqnarray*}} \newcommand{\eeQ}{\end{eqnarray*}}
\newcommand{\bi}{\begin{itemize}}    \newcommand{\ei}{\end{itemize}}
\newcommand{\bt}{\begin{tabular}}    \newcommand{\et}{\end{tabular}}

\author[F. Manfio]{Fernando Manfio}
\address{Institute of Mathematics and Computer Science \newline
\indent University of S\~ao Paulo, S\~ao Carlos, Brazil}
\email{manfio@icmc.usp.br}
\thanks{The first author is supported by Fapesp, grant \textbf{2019/23370-4} and the third author gratefully acknowledges the financial support from the Indian Institute of Technology Goa through Start-up grant \textbf{2021/SG/AU/043}.}

\author[J. Roth]{Julien Roth}
\address{Laboratoire d'Analyse et de Math\'ematiques Appliqu\'ees, \newline
\indent UPEM-UPEC, CNRS, F-77454 Marne-la-Vall\'ee, France}
\email{julien.roth@u-pem.fr}

\author[A. Upadhyay]{Abhitosh Upadhyay}
\address{School of Mathematice and Computer Science,\newline
\indent Indian Institute Of Technology Goa, Ponda, 403401, India}
\email{abhitosh@iitgoa.ac.in}

\title[Extrinsic eigenvalues upper bounds in weighted manifolds]
{Extrinsic eigenvalues upper bounds for submanifolds in weighted manifolds}

\begin{document}
 
\maketitle

\begin{abstract}
We prove Reilly-type upper bounds for divergence-type operators of the 
second order as well as for Steklov problems on submanifolds of Riemannian
manifolds of bounded sectional curvature endowed with a weighted measure.
\end{abstract}

\vspace{.2cm}

\vspace{2ex}


\section{Introduction}

Let $(M^n,g)$ be an $n$-dimensional compact, connected, oriented 
manifold without boundary, and consider an isometric immersion 
$X:M^n\to\R^{n+1}$ in the Euclidean space. The spectrum of the 
Laplacian of $(M,g)$ is an increasing sequence of real numbers
\[
0=\lambda_0(\Delta)<\lambda_1(\Delta)\leqslant\lambda_2(\Delta)
\leqslant\cdots\leqslant\lambda_k(\Delta)\leqslant\cdots\lgra+\infty.
\]
The eigenvalue $0$ (corresponding to constant functions) is simple and $\lambda_1(\Delta)$ is the first positive eigenvalue. In \cite{Re}, Reilly proved the following well-known upper bound for $\lambda_1(\Delta)$
\begin{equation}\label{Reilly1}
\lambda_1(\Delta)\leqslant\frac{n}{\vvol(M)}\int_MH^2dv_g,
\end{equation}
where $H$ is the mean curvature of the immersion. He also proved an 
analogous inequality involving the higher order mean curvatures. 
Namely, for $r\in\{1,\cdots,n\}$
\begin{equation}\label{Reillyr}
\lambda_1(\Delta)\left(\int_MH_{r-1}dv_g\right)^2\leqslant\vvol(M)\int_MH_r^2dv_g,
\end{equation}
where $H_r$ is the $r$-th mean curvature, defined by the $r$-th symmetric polynomial of the principal curvatures. Moreover, Reilly studied the equality cases and proved that equality in \eqref{Reilly1} or \eqref{Reillyr} is attained if and only if $X(M)$ is a geodesic sphere.

In the case of higher codimension, Reilly also proved that
\begin{equation}\label{Reillycodim1}
\lambda_1(\Delta)\leqslant\frac{n}{\vvol(M)}\int_M\|{\bf H}\|^2dv_g,
\end{equation}
where ${\bf H}$ is here the mean curvature vector, with equality if and only 
if $M$ is minimally immersed in a geodesic sphere.

The Reilly inequality can be easily extended to submanifolds of the sphere $\Ss^n$ using the canonical embedding of $\Ss^n$ into $\R^{n+1}$:
\begin{equation}\label{Reillycodim2}
\lambda_1(\Delta)\leqslant\frac{n}{\vvol(M)}\int_M(\|{\bf H}\|^2+1)dv_g.
\end{equation}
Moreover, El Soufi and Ilias \cite{EI} proved an analogue for submanifold of the hyperbolic space as
\begin{equation}\label{Reillycodim3}
\lambda_1(\Delta)\leqslant\frac{n}{\vvol(M)}\int_M(\|{\bf H}\|^2-1)dv_g.
\end{equation}
In case the ambient space has non-constant sectional curvature, Heintze 
\cite{He} proved the following weaker inequality
\begin{equation}\label{Heinzte}
\lambda_1(\Delta)\leqslant n(\|{\bf H}\|^2+\delta),
\end{equation}
where the ambient sectional curvature is bounded above by $\delta$.

On the other hand, more recently, in \cite{Ro2}, the second author prove
the following general inequality
\beqt\label{generalLT}
\lambda_1(L_{T,f})\left(\int_M\trace(S)\mu_f\right)^2\leqslant\left(\int_M\trace(T)\mu_f\right)\int_M\left(\|H_S\|^2+\|S\nabla f\|^2\right)\mu_f,
\eeqt
where $\mu_f=e^{-f}dv_g$ is the weighted measure of $(M,g)$ endowed 
with the density $e^{-f}$, $T,S$ are two symmetric, free-divergence 
$(1,1)$-tensors with $T$ positive definite, and $L_{T,f}$ is the second 
order differential operator defined for any smooth function $u$ on 
$M$ by
\[
L_{T,f}=-\ddiv(T\nabla u)+\langle \nabla f,T\nabla u\rangle.
\]
When $f=0$ and the tensor $S$ and $T$ are associated with higher 
order mean curvatures $H_s$ and $H_r$, we recover the inequality 
of Alias and Malacarn\'e \cite{AM} and, in particular, Reilly's inequality 
\eqref{Reillyr} if $r=0$.

\vspace{.2cm}

The first result of this paper gives upper bounds for the first eigenvalue 
of the operator $L_{T,f}$ for submanifolds of Riemannian manifolds with 
sectional curvature bounded by above which generalizes inequality \eqref{generalLT} in the non-constant curvature case. Namely, we prove 
the following.

\begin{thm}\label{thm1}
Let $(\bar{M}^{n+p},\bar{g},\bar{\mu}_{f})$ be a weighted Riemannian 
manifold with sectional curvature $sect_{\bar{M}}\leqslant\delta$ and 
$\bar{\mu}_{f}=e^{-f}dv_{\bar{g}}$. Let $(M,g)$ be a closed Riemannian 
manifold isometrically immersed into $(\bar{M}^{n+p},\bar{g}$) by $X$. 
We endow $M$ with the weighted measure $\mu_{f}=e^{-f}dv_g$. Let 
$T$ be a positive definite $(1,1)$-tensor on $M$ and denote by 
$\lambda_1$ the first positive eigenvalue of the operator $L_{T,f}$.
\begin{enumerate}

\item If $\delta\leqslant 0$, then
\[
\lambda_1\leqslant\sup_M\left[\delta\trace(T)+
\sup_M\left( \frac{\|H_T-T\nabla f\|}{\trace(S)}\right)\|H_S-S\nabla f\|\right].
\]

\item If $\delta>0$ and $X(M)$ is contained in a geodesic ball of radius 
$\frac{\pi}{4\sqrt{\delta}}$, 
\[
\lambda_1\leqslant\dfrac{\left( \displaystyle\int_M\trace(T)\mu_{f}\right)}{V_f(M)}\left(\delta  +\frac{ \displaystyle\int_M\|H_S-S(\nabla f)\|^2\mu_{f}}{V_f(M)
\inf\left(\trace(S)^2\right)}\right).
\]
\end{enumerate}
\end{thm}

The second eigenvalue problem that we consider in this paper is the 
Steklov problem associated with the operator $L_{T,f}$ on a submanifold 
$\Omega$ with non-empty boundary $\partial\Omega=M$ of a Riemannian
manifold with sectional curvature bounded by above. We can consider 
the following generalized weighted Steklov problem
\begin{equation} \label{SPLT}
\displaystyle
\left\{\begin{array}{ll}
L_{T,f} u=0&\text{on}\ \Omega,\\ \\
\dfrac{\partial u}{\partial\nu_T}=\sigma u&\text{on}\ M=\partial \Omega,
\end{array}
\right.
\end{equation}
where 
$\frac{\partial u}{\partial\nu_T}=\langle T(\widetilde{\nabla}u),\nu\rangle$. 
In the case where $f$ is constant, the operator $L_{T,f}$ is of particular 
interest for the study of $r$-stability when $T=T_r$ is the tensor associated 
with $r$-th mean curvature (see \cite{ACR} for instance). More precisely, 
from \cite{Au}, we know that this problem  \eqref{SPLT} has a discrete 
nonnegative spectrum and we denote by $\sigma_1$ its first eigenvalue. 
In \cite{Ro}, the second author has obtained upper bounds for this 
problem for domains of a manifold lying in a Euclidean space. 
Namely, he has proved
\[
\sigma_{1}\left(\int_{ M}\trace(S)\widetilde{\mu}_f\right)^2\leqslant 
\left( \int_{\Omega}\trace(T)\mu_f\right)\int_{ M}\big(|| H_S||^2+
||S\nabla f||^2\big)\widetilde{\mu}_f,
\]
where $\mu_f=e^{-f}dv_g$ and $\widetilde{\mu}_f=e^{-f}dv_{\widetilde{g}}$ 
are respectively, the weighted measures of $(\Omega,g)$ and 
$(M,\widetilde{g})$ endowed with the density $e^{-f}$ and where $T$ and 
$S$ are two symmetric, free-divergence $(1,1)$-tensors with $T$ positive 
definite. Note that without density and for $T=\iid$, this inequality has 
been proven by Ilias and Makhoul \cite{IM}. 

\vspace{.2cm}

The second result of the present paper gives a generalization of this
estimate when the manifold with boundary $(M,g)$ is immersed into 
an ambient Riemannian manifold of sectional curvature bounded by 
above. Namely, we prove the following.

\begin{thm}\label{thm2}
Let $(\bar{M}^{n+p},\bar{g},\bar{\mu}_{f})$ be a weighted Riemannian 
manifold with sectional curvature $sect_{\bar{M}}\leqslant\delta$ and 
$\bar{\mu}_{f}=e^{-f}dv_{\bar{g}}$. Let $(\Omega,g)$ be a compact 
Riemannian manifold with non-empty boundary $M$ isometrically 
immersed into $(\bar{M}^{n+p},\bar{g})$ by $X$. We endow $\Omega$ 
and $M$, respectively with the weighted measure $\mu_{f}=e^{-f}dv_g$ 
and $\tilde{\mu}_f=e^{-f}dv_{\tilde{g}}$, where $\tilde{g}$ is the induced 
metric on $M$. Let $T$, $S$ be a symmetric, divergence-free and 
positive definite $(1,1)$-tensors on $\Omega$ and $M$, respectively,
and denote by $\sigma_1$ the first eigenvalue of the Steklov problem 
\eqref{SPLT}.
\begin{enumerate}
\item If $\delta\leqslant 0$ and $X(\Omega)$ is contained in the  geodesic 
ball $B(p,R)$ of radius $R$, where $p$ is the center of mass of $M$ for 
the measure $\widetilde{\mu}_f$, then 

\begin{eqnarray*}
\sigma_{1}\leqslant \sup_{\Omega}\left[ \delta\trace(T)+\sup_{\Omega}\left( \frac{\|H_T-T(\overline{\nabla} f)\|}{\trace(T)}\right)\|H_T-T(\overline{\nabla} f)\|\right]\\ \times\left[\delta+\dfrac{\sup_M\|H_S-S(\nabla f)\|^2}{\inf_M(\trace(S))^2}\right]\frac{V_f(\Omega)}{V_f(M)}\sd^2(R).
\end{eqnarray*}
\item If $\delta>0$ and $X(\Omega)$ is contained in a geodesic ball of radius $\frac{\pi}{4\sqrt{\delta}}$, then
$$\sigma_{1}\leqslant \frac{\displaystyle \int_{\Omega}\trace(T)\mu_{f}}{V_f(M)}\left(\delta  +\frac{ \displaystyle\int_{M}\|H_S-S(\nabla f)\|^2\widetilde{\mu}_f}{V_f(M)\inf\left(\trace(S)^2\right)}\right).$$
\end{enumerate}
\end{thm}

Finally, we will consider the so-called eigenvalue problem for Wentzell 
boundary conditions
\begin{equation} \label{SWP} 
\displaystyle
\left\{\begin{array}{rcll}
\Delta u &=& 0 & \text{in} \ \Omega \\
-b\widetilde{\Delta} u-\dfrac{\partial u}{\partial\nu} &=& \alpha u & \text{on} \ M,
\end{array}
\right.
\end{equation}
where $b$ is a given positive constant, $\Omega$ is a submanifold with
non-empty boundary $\partial\Omega=M$ of a Riemannian manifold $M$
with sectional curvature bounded by above, and $\Delta$, 
$\widetilde{\Delta}$ denote the Laplacians on $\Omega$ and $M$, 
respectively. It is clear that if $b=0$, then we recover the classical 
Steklov problem. The spectrum of this problem is an increasing 
sequence (see \cite{DKL}) with $0$ as first eigenvalue which is 
simple and the corresponding eigenfunctions are the constant 
ones. We denote by $\alpha_1$ the first positive eigenvalue. 
In \cite{Ro2}, the second author proved the following estimate 
when $\Omega$ is a submanifold of the Euclidean space 
$\mathbb{R}^n$ 
\[
\alpha_1\left(\int_{\partial M}\trace(S)dv_g\right)^2\leqslant \Big(nV(M)+b(n-1)V(\partial M)\Big)\left(\int_{\partial M}\| H_S\|^2dv_g\right).
\]

\vspace{.2cm}

In the following theorem, we obtain a comparable estimate when the 
ambient space is of bounded sectional curvature. Namely, we prove the following.

\begin{thm}\label{thm3}
Let $(\bar{M}^{n+p},\bar{g})$ be a Riemannian manifold with sectional curvature $sect_{\bar{M}}\leqslant\delta$. Let $(\Omega,g)$ be a compact Riemannian manifold with non-empty boundary  $M$ isometrically immersed into $(\bar{M}^{n+p},\bar{g})$ by $X$. We denote by $\widetilde{g}$ the induced metric on $M$. Let $S$ be a symmetric, divergence-free and positive definite $(1,1)$-tensor on $M$ and denote by $\alpha_1$ the first eigenvalue of the Steklov-Wentzell problem \eqref{SWP}.
\begin{enumerate}
\item If $\delta\leqslant 0$ and $X(\Omega)$ is contained in the geodesic 
ball $B(p,R)$ of radius $R$, where $p$ is the center of mass of $\overline{M}$, then 
$$\alpha_1\leqslant \Big[n\frac{V(\Omega)}{V(M)}+b(n-1)-\delta\sd^2(R)\left(\frac{V(\Omega)}{V(M)}+b\right)\Big]\left(\delta+\dfrac{\sup_M\|H_S\|^2}{\inf_M(\trace(S))^2}\right).$$
\item If $\delta>0$ and so $X(\Omega)$ is contained in a geodesic ball of radius $\frac{\pi}{4\sqrt{\delta}}$, then
\[
\alpha_1\leqslant \left( n\frac{V(\Omega)}{V(M)}+b(n-1)\right)\left(\delta  +
\frac{ \displaystyle\int_{M}\|H_S\|^2dv_{\widetilde{g}}}{V(M)
\inf\left(\trace(S)^2\right)}\right).
\]
\end{enumerate}
\end{thm}

\section{Preliminaries}

Let $(\bar{M}^{n+p},\bar{g},\bar{\mu}_{f})$ be a weighted Riemannian 
manifold with sectional curvature $sect_{\bar{M}}\leqslant\delta$ and 
weighted measure $\bar{\mu}_{f}=e^{-f}dv_{\bar{g}}$. Let $p$ a fixed 
point in $\bar{M}$, we denote by $r(x)$ the geodesic distance between 
$x$ and $p$. Moreover, we define the vector field $X$ by 
$X(x):=\sd(r(x))(\bar{\nabla} r)(x)$, $\sd$ is the function defined by
\[ \sd(r)=\left\{ \begin{array}{lll}

\frac{1}{\sqrt{\delta}}\sin(\sqrt{\delta}r) &\; \text{if} \;\; \delta>0\\

r &\; \text{if}\;\; \delta=0 \\

\frac{1}{\sqrt{|\delta|}}\sinh(\sqrt{|\delta|}r) &\; \text{if} \;\; \delta<0.\ \ \end{array} \right.\] 
We also define 
\[ \cd(r)=\left\{ \begin{array}{lll}

\cos(\sqrt{\delta}r) &\; \text{if} \;\; \delta>0\\

1 &\; \text{if}\;\; \delta=0 \\

\cosh(\sqrt{|\delta|}r) &\; \text{if} \;\; \delta<0.\ \ \end{array} \right.
\] 
Hence, we have
\[
\cd^2+\delta\sd^2=1, \quad \sd'=\cd \quad\text{and}\quad \cd'=-\delta\sd.
\]

In addition, let $(M^n,g)$ be a closed Riemannian manifold isometrically 
immersed into $(\bar{M}^{n+p},\bar{g})$ by $\phi$. If $\delta>0$, then we
assume that $\phi(M)$ is contained in a geodesic ball of radius $\frac{\pi}
{2\sqrt{\delta}}$. We endow $M$ with the weighted measure 
$\mu_{f}=e^{-f}dv_{g}$. We can define on $M$ a divergence associated 
with the volume form $\mu_{f}=e^{-f}dv_g$ by
\[
\ddiv_fY=\ddiv Y-\langle \nabla f,Y\rangle
\]
or, equivalently,
\[
d(\iota_{Y}\mu_{f} )=\ddiv_{f}(Y)\mu_{f},
\]
where $\nabla$ is the gradient on $\Sigma$, that is, the projection on 
$T\Sigma$ of the gradient $\bar{\nabla}$ on $\bar{M}$. We call it the 
$f$-divergence. We recall briefly some basic facts about the 
$f$-divergence. In the case where $\Sigma$ is closed, we first have 
the weighted version of the divergence theorem:
\begin{equation}\label{intdiv}
\int_{\Sigma}\ddiv_fY~\mu_{f}=0,
\end{equation}
for any vector field $Y$ on $\Sigma$. From this, we deduce easily the 
integration by parts formula 
\begin{equation}\label{IPP}
\int_{\Sigma} u~\ddiv_fY~\mu_{f}=-\int_{\Sigma}\langle \nabla u,Y\rangle~\mu_{f},
\end{equation}
for any smooth function $u$  and any vector field $Y$ on $\Sigma$.
First, we prove the following elementary lemma which generalize in 
non constant curvature the classical Hsiung-Minkowski formula (see 
\cite{Gr1,Ro} for instance).

\begin{lemma}\label{lem1}
Let $T$ be a symmetric divergence-free positive $(1,1)$-tensor on $M$. 
Then the following hold
\begin{enumerate}
\item $\ddiv_f\left(TX^{\top}\right)\geqslant \trace(T)\cd+\langle X,H_T-T(\nabla f)\rangle$.
\item $\displaystyle\int_M\trace(T)\cd \mu_{f}\leqslant-\displaystyle\int_M\langle X,H_T-T(\nabla f)\rangle \mu_{f}$.
\item $\delta\displaystyle\int_M\langle TX^{\top},X^{\top}\rangle \mu_{f}\geqslant \displaystyle\int_M\trace(T)\cd^2\mu_{f}-\int_M\|H_T-T(\nabla f)\|\sd\cd\mu_{f}$. 
\end{enumerate}
\end{lemma}
\begin{proof}
The proof is a straightforward consequence of the analogue non-weighted
result proven by Grosjean. Namely, in \cite{Gr1}, the author has shown 
that
\[
\ddiv_M(TX^{\top})\geqslant \trace(T)\cd(r)+\langle X,H_T\rangle.
\]
Hence, from the definition of the $f$-divergence, we have
\beQ
\ddiv_f(TX^{\top})&=&\ddiv (TX^{\top})-\langle \nabla f,TX^{\top}\rangle\\
&\geqslant&\trace(T)\cd(r)+\langle X,H_T-T(\nabla f)\rangle,
\eeQ
and this proves part $(1)$. For the second part, we integrate the last 
inequality with respect to the measure $\mu_{f}$ and we get immediately 
\[
\displaystyle\int_M\trace(T)\cd \mu_{f}\leqslant-\displaystyle
\int_M\langle X,H_T-T(\nabla f)\rangle \mu_{f},
\]
since
\[
\displaystyle\int_M\ddiv_f(TX^{\top})\mu_{f}=0.
\]
Finally, for the last part, if $\delta=0$, then $\cd=1$ and we get directly the 
conclusion from the second one by using 
\[
|\langle X,H_T-T(\nabla f)\rangle|\leqslant ||X||\cdot||H_T-T(\nabla f)||=
\sd||H_T-T(\nabla f)||.
\]
If $\delta\neq0$, since $X^{\top}=\sd(r)\nabla r=-\delta \nabla \cd(r)$, then 
we have 
\beQ
\delta\displaystyle\int_M\langle TX^{\top},X^{\top}\rangle \mu_{f}&=&\frac{1}{\delta}\int_M\langle T(\nabla \cd,\nabla \cd)\mu_{f}\\
&=&-\frac{1}{\delta}\int_M\ddiv_f(T\nabla\cd)\cd\mu_{f}\\
&=&\int_M\ddiv_f(X^{\top})\cd\mu_{f}\\
&\geqslant&\int_M\trace(T)\cd^2\mu_{f}-\int_M\|H_T-T(\nabla f)\|\sd\cd\mu_{f},
\eeQ
where we have used the first part of the lemma and the well-known 
Cauchy-Schwarz inequality.
\end{proof}

\section{Two key lemma}

In this section we will prove two basic key lemma that will be used 
throughout the paper.

\begin{lemma}\label{lem31}
Let $(\bar{M}^{n+p},\bar{g},\bar{\mu}_{f})$ be a weighted Riemannian 
manifold with sectional curvature $sect_{\bar{M}}\leqslant\delta\leqslant0$,
and $\bar{\mu}_{f}=e^{-f}dv_{\bar{g}}$. Let $(M,g)$ be a closed Riemannian
manifold isometrically immersed into $(\bar{M}^{n+p},\bar{g})$, and we 
endow $M$ with the weighted measure $\mu_{f}=e^{-f}dv_g$. Let $S$ be a
symmetric, divergence-free and positive definite $(1,1)$-tensor on $M$. 
Then, we have
\[
\frac{\displaystyle\int_M\|X\|^2\mu_{f}}{V_f(M)}\geqslant\dfrac{1}
{\delta+\dfrac{\|H_S-S(\nabla f)\|^2_{\infty}}{\inf(\trace(S))^2}}.
\]
\end{lemma}
\begin{proof}
We have
\[
\begin{array}{ll}
\displaystyle\int_M\left(\trace(S)-\delta\langle SX^{\top},X^{\top}\rangle\right) \mu_{f}=\int_M\left(\trace(S)-\ddiv(SX^{\top})\cd(r)\right)\mu_{f}\\
\displaystyle\leqslant\int_M\left(\trace(S)-\trace(S)\cd^2(r)-\langle H_S-S(\nabla f),X\rangle\cd(r)\right)\mu_{f}\\
\displaystyle\leqslant\int_M\left(\delta\trace(S)\sd^2(r)-\langle H_S-S(\nabla f),X\rangle\cd(r)\right)\mu_{f}\\
\displaystyle\leqslant\int_M\left(\delta\trace(S)\sd^2(r)+\frac{\|H_S-S(\nabla f)\|_{\infty}}{\inf(\trace(S))}\int_M\trace(S)\sd(r)\cd(r)\right)\mu_{f} \\

\displaystyle\leqslant\frac{\|H_S-S(\nabla f)\|_{\infty}}{\inf(\trace(S))}
\int_M\left(\sd(r)\ddiv(SX^{\top})
 -\sd(r)\langle H_S-S(\nabla f),X\rangle\right)\mu_{f}  \vspace{.2cm} \\
\hspace{.5cm} + \delta\inf(\trace(S))\int_M\sd^2(r)\mu_{f} \vspace{.2cm} \\
 
\displaystyle\leqslant-\frac{\|H_S-S(\nabla f)\|_{\infty}}{\inf(\trace(S))}
\int_M\left(  \cd(r)\sd(r)\langle S\nabla^Mr,\nabla^Mr\rangle+\sd^2(r)\langle 
H_S-S(\nabla f),\nabla^Mr\rangle\right)\mu_{f} \vspace{.2cm} \\
\hspace{.5cm}
+ \delta\inf(\trace(S))\int_M\sd^2(r)\mu_{f}
\end{array}
\]
\[
\begin{array}{ll}
\displaystyle\leqslant\left(\delta\inf(\trace(S))+\frac{\|H_S-S(\nabla f)\|_{\infty}^2}{\inf(\trace(S))}\right)\int_M\sd^2(r)\mu_{f} \vspace{.2cm} \\
\hspace{2cm}
-\displaystyle\frac{\|H_S-S(\nabla f)\|_{\infty}}{\inf(\trace(S))}\int_M\cd(r)\sd(r)\langle S\nabla^Mr,\nabla^Mr\rangle \mu_{f}.
\end{array}
\]
Hence, we get
\beQ
\int_M\trace(S)\mu_{f}&\leqslant&  \left(\delta\inf(\trace(S))+\frac{\|H_S-S(\nabla f)\|_{\infty}^2}{\inf(\trace(S))}\right)\int_M\sd^2(r)\mu_{f} \\
&&-\frac{\|H_S-S(\nabla f)\|_{\infty}}{\inf(\trace(S))}\int_M\cd(r)\sd(r)\langle S\nabla^Mr,\nabla^Mr\rangle \mu_{f} \\
&& +\delta\int_M\sd^2(r)\langle S\nabla^Mr,\nabla^Mr\rangle \mu_{f}\\
&\leqslant&\left(\delta\inf(\trace(S))+\frac{\|H_S-S(\nabla f)\|_{\infty}^2}{\inf(\trace(S))}\right)\int_M\sd^2(r)\mu_{f}\\
&&+\int_M\left(\delta\sd^2(r)-\cd(r)\sd(r)\frac{\|H_S-S(\nabla f)\|_{\infty}}{\inf(\trace(S))}\right)\langle S\nabla^Mr,\nabla^Mr\rangle \mu_{f}.
\eeQ
Since $\delta\leqslant0$, the second term of the right hand side is nonpositive,
and thus we get
\begin{eqnarray*}
\inf(\trace(S))V_f(M) &\leqslant& \int_M\trace(S)\mu_{f} \\ 
&\leqslant& \left(\delta\inf(\trace(S))+\frac{\|H_S-S(\nabla f)\|_{\infty}^2}{\inf(\trace(S))}\right)\int_M\sd^2(r)\mu_{f},
\end{eqnarray*}
which gives immediately  the result since $\|X\|=\sd(r)$.
\end{proof}

\begin{lemma} \label{lem2}
Let $(\bar{M}^{n+p},\bar{g},\bar{\mu}_{f})$ be a weighted Riemannian 
manifold with sectional curvature $sect_{\bar{M}}\leqslant\delta$, with
$\delta>0$, and $\bar{\mu}_{f}=e^{-f}dv_{\bar{g}}$. Let $(M,g)$ be a 
closed Riemannian manifold isometrically immersed into 
$(\bar{M}^{n+p},\bar{g})$ by $X$ so that $X(M)$ is contained in a 
geodesic ball of radius $\frac{\pi}{2\sqrt{\delta}}$. We endow $M$ with 
the weighted measure $\mu_{f}=e^{-f}dv_g$. Let $S$ be a symmetric, 
divergence-free and positive definite $(1,1)$-tensor on $M$. Then, 
we have
\[
1-\left(\frac {\displaystyle\int_M\cd(r)\mu_{f}}{V_f(M)}\right)^2
\geqslant\frac{1}{1+\dfrac{\displaystyle\int_M\|H_S-
S(\nabla f)\|^2\mu_{f}}{\delta \inf(\trace(S))^2V_f(M)}}.
\]
\end{lemma}
\begin{proof}
For a sake of compactness, we will write 
$$\alpha=\frac {\displaystyle\int_M\cd(r)\mu_{f}}{V_f(M)}\quad\text{and}\quad \beta=1+\dfrac{\displaystyle\int_M\|H_S-S(\nabla f)\|^2\mu_{f}}{\delta \inf(\trace(S))^2V_f(M)}.$$
We thus have to show that $(1-\alpha^2)\beta\geqslant 1$. We have
\beQ
(1-\alpha^2)\beta&=&\beta-\left(\frac {\displaystyle\int_M\cd(r)\mu_{f}}{V_f(M)}\right)^2-\left(\frac {\displaystyle\int_M\cd(r)\mu_{f}}{V_f(M)}\right)^2\dfrac{\displaystyle\int_M\|H_S-S(\nabla f)\|^2\mu_{f}}{\delta \inf(\trace(S))^2V_f(M)}\\
&\geqslant&\beta-\left(\frac {\displaystyle\int_M\trace(S)\cd(r)\mu_{f}}{\inf(\trace(S))V_f(M)}\right)^2-\left(\frac {\displaystyle\int_M\cd(r)^2\mu_{f}}{V_f(M)}\right)\dfrac{\displaystyle\int_M\|H_S-S(\nabla f)\|^2\mu_{f}}{\delta \inf(\trace(S))^2V_f(M)}\\
&\geqslant&\beta-\left(\frac{\displaystyle \int_M \sd(r)\|H_S-S(\nabla f)\|\mu_{f}}{\inf(\trace(S))V_f(M)}\right)^2-\left(\frac {\displaystyle\int_M\cd(r)^2\mu_{f}}{V_f(M)}\right)\dfrac{\displaystyle\int_M\|H_S-S(\nabla f)\|^2\mu_{f}}{\delta \inf(\trace(S))^2V_f(M)}\\
&\geqslant&\beta-\dfrac{ \left(\displaystyle\int_M \sd^2(r)\mu_{f}\right)\left(\displaystyle\int_M\|H_S-S(\nabla f)\|^2\mu_{f}\right)}
{\inf(\trace(S))^2V_f(M)^2} \\
&&-\left(\frac {\displaystyle\int_M\cd(r)^2\mu_{f}}{V_f(M)}\right)\dfrac{\displaystyle\int_M\|H_S-S(\nabla f)\|^2\mu_{f}}{\delta \inf(\trace(S))^2V_f(M)}\\
&\geqslant&\beta-\dfrac{ \displaystyle\int_M\|H_S-S(\nabla f)\|^2\mu_{f}}{\delta\inf(\trace(S))^2V_f(M)^2}\left(\int_M(\sd^2(r)+\delta^2\cd^2(r))\mu_{f}\right)\\
&=&\beta-\dfrac{ \displaystyle\int_M\|H_S-S(\nabla f)\|^2\mu_{f}}{\delta\inf(\trace(S))^2V_f(M)}\\
&=&1,
\eeQ
and this concludes the proof.
\end{proof}

\section{Proofs of the Theorems}

\begin{proof}[Proof of Theorem \ref{thm1}]
Case $\delta\leqslant0$.
Let $p\in\overline{M}$ be a fixed point and let $\{x_1,\cdots,x_N\}$ be the 
normal coordinates of $\overline{M}$ centered at $p$. For any 
$x\in\overline{M}$, $r(x)$ is the geodesic distance between $p$ and $x$ 
over $\overline{M}$. We want to use as test functions the functions
\[
\frac{\sd(r)}{r}x_i,
\]
for $1\leqslant i\leqslant N$. For this purpose, we will 
choose $p$ as the center of mass of $M$ with respect for the measure 
$\mu_f$, that is, $p$ is the only point in $M$ so that 
$$
\int_M\dfrac{\sd(r)}{r}x_i\mu_f=0,
$$
for any $i\in\{1,\cdots,N\}$. Note at that point that we assume that $M$ 
is contained in a ball of radius $\dfrac{\pi}{4\sqrt{\delta}}$ when $\delta$ 
is positive allows us to ensure that $M$ is contained in a ball of radius 
$\dfrac{\pi}{4\sqrt{\delta}}$ centered at $p$. This holds also for Theorems 
\ref{thm2} and \ref{thm3}.
These functions are candidates for test functions since they are $L^2(\mu_f)$-orthogonal to the constant functions which are the eigenfunctions for the first eigenvalue $\lambda_0=0$. Thus, we have
\beqt\label{quotient1}
\lambda_1\int_M\sum_{i=1}^N\frac{\sd^2(r)}{r^2}x_i^2\mu_f\leqslant\int_M\sum_{i=1}^N\left\langle T\nabla\left(\frac{\sd(r)}{r}x_i\right),\nabla\left(\frac{\sd(r)}{r}x_i\right)\right\rangle\mu_f.
\eeqt
We recall that Grosjean proved in \cite[Lemma 2.1]{Gr1} that 
\beqt\label{ineqgrosjean}
\sum_{i=1}^N\left\langle T\nabla\left(\frac{\sd(r)}{r}x_i\right),\nabla\left(\frac{\sd(r)}{r}x_i\right)\right\rangle\leqslant \trace(T)-\delta\left \langle TX^{\top},X^{\top}\right\rangle.
\eeqt
Moreover, by Lemma \ref{lem1}, item (3), we have
$$\delta\displaystyle\int_M\langle TX^{\top},X^{\top}\rangle \mu_{f}\geqslant \displaystyle\int_M\trace(T)\cd^2(r)\mu_{f}-\int_M\|H_T-T(\nabla f)\|\sd(r)\cd(r)\mu_{f}$$
which together with \eqref{quotient1} and \eqref{ineqgrosjean} gives
\beQ
\lambda_1\int_M\sd^2(r)\mu_f&\leqslant&\int_M\Big(\trace(T)-\cd^2(r)\trace(T)+\|H_T-T(\nabla f)\|\sd(r)\cd(r)\Big)\mu_{f}\\
&\leqslant&\delta\int_M\sd^2(r)\trace(T)\mu_f+\int_M\|H_T-T(\nabla f)\|\sd(r)\cd(r)\mu_{f},
\eeQ
where we have used $\cd^2+\delta\sd^2=1$. Hence, we obtain
\beq\label{majlambda1}
\lambda_1\int_M\sd^2(r)\mu_f&\leqslant&\delta\int_M\sd^2(r)\trace(T)\mu_f\\
&&+\sup_{M}\left(\frac{\|H_T-T(\nabla f)\|}{\trace(S)}\right)\int_M\trace(S)\sd(r)\cd(r)\mu_{f}\nonumber.
\eeq

Now, we claim the following.

\begin{lemma} \label{lemsd}
We have
$$\displaystyle\int_M\trace(S)\sd(r)\cd(r)\mu_f\leqslant \int_M\|H_S-S(\nabla f)\|\sd^2(r)\mu_f.$$
\end{lemma}

Reporting this into \eqref{majlambda1}, we get
\[
\lambda_1\int_M\sd^2(r)\mu_f \leqslant \int_M\Bigg[\delta\trace(T)+\sup_{M}\left(\frac{\|H_T-T(\nabla f)\|}{\trace(S)}\right)\|H_S-S(\nabla f)\|\Bigg]\sd^2(r)\mu_{f}.
\]
which gives immediately the desired estimate
$$\lambda_1\leqslant\sup_M\left[\delta\trace(T)+\sup_M\left( \frac{\|H_T-T\nabla f\|}{\trace(S)}\right)\|H_S-S\nabla f\|\right].$$
This concludes the proof for the case $\delta<0$, up to the proof of the 
lemma that we give now.

\begin{proof}[Proof of Lemma \ref{lemsd}]
Multiplying the first part of Lemma \ref{lem1} for the tensor $S$ by $\sd(r)$, 
we get
\[
\ddiv_f(SX^{\top})\sd(r)\geqslant \trace(S)\cd(r)\sd(r)-
\langle X,H_S-S(\nabla f)\rangle,
\]
and thus, by the Cauchy-Schwarz inequality, we have
\[
\ddiv_f(SX^{\top})\sd(r)\geqslant \trace(S)\cd(r)\sd(r)-\|H_S-S(\nabla f)\|\sd(r).
\]
Now, integrating this relation and using the integration by parts in the
formula \eqref{IPP}, we get
\beQ
\int_M\trace(S)\cd(r)\sd(r)\mu_f-\int_M\|H_S-S(\nabla f)\|\sd(r)\mu_f&\leqslant& -\int_M\langle \nabla \sd(r),SX^{\top}\rangle \mu_f\\
&\leqslant&-\int_M\cd(r)\sd(r)\langle \nabla r,S\nabla\rangle\\
&\leqslant&0,
\eeQ
since $S$ is positive. This concludes the proof of Lemma \ref{lemsd}.
\end{proof}
Case $\delta>0$.
Like in the case $\delta<0$, we use $\frac{\sd(r)}{r}x_i$, $1\leqslant i\leqslant N$, as test functions. Using \eqref{ineqgrosjean} again, we get
\beqt\label{majlambda13}
\lambda_1\int_M\sd^2(r)\mu_f\leqslant\int_M\left(\trace(T)-\delta\left \langle TX^{\top},X^{\top}\right\rangle\right)\mu_f.
\eeqt
On the other, we use another test function in this case, namely $\cd(r)-\overline{\cd}$ where for more convenience, we have denoted by $\overline{\cd}$ the mean value of $\cd(r)$, that is $\overline{\cd}=\dfrac{1}{V_f(M)}\int_M\cd(r)\mu_f$. 
Here again, this function is $L^2(\mu_f)$-orthogonal to the constant functions,
so it is a candidate for being a test function. Hence, we have
\beQ
\lambda_1\int_M\left(\cd(r)-\overline{\cd}\right)^2\mu_f&\leqslant&\int_M\left\langle T\nabla \left(\cd(r)-\overline{\cd}\right), \nabla\left(\cd(r)-\overline{\cd}\right)\right\rangle\mu_f\\
&\leqslant&\int_M\left\langle T\nabla\cd(r), \nabla\cd(r)\right\rangle\mu_f\\
&\leqslant&\delta^2\int_M\sd^2(r)\left\langle T\nabla r,\nabla r\right\rangle\mu_f\\
&\leqslant&\delta^2\int_M\left\langle TX^{\top},X^{\top}\right\rangle\mu_f.
\eeQ
From this, we deduce immediately that
\beqt\label{majlambda14}
\lambda_1\int_M\cd^2(r)\mu_f\leqslant\delta^2\int_M\left\langle TX^{\top},X^{\top}\right\rangle+\frac{\lambda_1}{V_f(M)}\left(\int_M\cd(r)\mu_f\right)^2.
\eeqt
Now, using the fact that $\cd^2+\delta\sd^2=1$, \eqref{majlambda14} plus $\delta$ times \eqref{majlambda13} gives
$$ \lambda_1V_f(M)\leqslant \delta\int_M\trace(T)\mu_f+\left(\dfrac{\displaystyle\int_M\cd(r)\mu_f}{V_f(M)}\right)^2$$
and thus
\[
\lambda_1V_f(M)\left( 1-\left(\frac{\int_M\cd(r)\mu_f}{V_f(M)}\right)^2\right)
\leqslant\delta\int_M\trace(T)\mu_f.
\]
Now, we conclude by using Lemma \ref{lem2} to get the desired upper 
bound
\[
\lambda_1\leqslant\dfrac{\left( \displaystyle\int_M\trace(T)
\mu_{f}\right)}{V_f(M)}\left(\delta  +\frac{ \displaystyle
\int_M\|H_S-S(\nabla f)\|^2\mu_{f}}{V_f(M)\inf\left(\trace(S)^2\right)}\right),
\]
and this concludes the proof of Theorem \ref{thm1}.
\end{proof}

\begin{proof}[Proof of Theorem \ref{thm2}]
{\it Case $\delta\leqslant0$.} Like in the proof of Theorem \ref{thm1}, we 
will consider $p\in\overline{M}$ as the center of mass of $\Omega$, 
$\{x_1,\ldots,x_N\}$ the normal coordinates of $\overline{M}$ centered
at $p$ and by $r(x)$ the geodesic distance (on $\overline{M}$) between 
$x$ and $p$, for any $x\in\overline{M}$. By the choice of $p$, we have 
\[
\int_M\dfrac{\sd(r)}{r}x_i\widetilde{\mu}_f=0,
\]
for any $i\in\{1,\ldots,N\}$, and we can use the functions 
$\frac{\sd(r)}{r}x_i$ as test functions in the variational characterization
of $\sigma_1$. Thus, we have
\begin{eqnarray} \label{quotientsigma1}
\begin{aligned}
\sigma_1\int_M\sum_{i=1}^N\frac{\sd^2(r)}{r^2}x_i^2\widetilde{\mu}_f
\leqslant&\int_{\Omega}\sum_{i=1}^N\left\langle T\nabla\left(\frac{\sd(r)}{r}x_i\right),\nabla\left(\frac{\sd(r)}{r}x_i\right)\right\rangle\mu_f \\
\leqslant&\int_{\Omega} \left(\trace(T)-\delta\left \langle TX^{\top},X^{\top}\right\rangle\right)\mu_f,
\end{aligned}
\end{eqnarray}
where we have used inequality \eqref{ineqgrosjean} to get the second line. 
Here, $X=\sd(r)\overline{\nabla} r$ and $X^{\top}=\sd(r)\nabla r$ is the
tangent component of $X$ to $\Omega$. On the other hand, by item (3) 
of Lemma \ref{lem1} applied to $\overline{M}$, we have
\[
\delta\displaystyle\int_{\Omega}\langle TX^{\top},X^{\top}\rangle \mu_{f}
\geqslant \displaystyle\int_{\Omega}\trace(T)\cd^2\mu_{f}-
\int_{\Omega}\|H_T-T(\nabla f)\|\sd\cd\mu_{f}.
\]
Reporting into \eqref{quotientsigma1}, we have
\beQ
\sigma_1\int_M\sd^2(r)\widetilde{\mu}_f&\leqslant&\int_{\Omega}\Big(\trace(T)-\cd^2(r)\trace(T)+\|H_T-T(\nabla f)\|\sd(r)\cd(r)\Big)\mu_{f}\\
&\leqslant&\delta\int_{\Omega}\sd^2(r)\trace(T)\mu_f+\int_{\Omega}\|H_T-T(\overline{\nabla} f)\|\sd(r)\cd(r)\mu_{f},\\
&\leqslant&\delta\int_{\Omega}\sd^2(r)\trace(T)\mu_f+\sup_{\Omega}\left(\frac{\|H_T-T(\overline{\nabla} f)\|}{\trace(T)}\right)\int_{\Omega}\trace(T)\sd(r)\cd(r)\mu_{f}.
\eeQ
From Lemma \ref{lemsd} applied on $\Omega$ for the tensor $T$, we have
$$\displaystyle\int_{\Omega}\trace(T)\sd(r)\cd(r)\mu_f\leqslant \int_{\Omega}\|H_T-T(\nabla f)\|\sd^2(r)\mu_f,$$
which yields to
\beQ
\sigma_1\int_M\sd^2(r)\widetilde{\mu}_f&\leqslant&\int_{\Omega}\sd^2(r)\left[\delta\trace(T)+\sup_{\Omega}\left( \frac{\|H_T-T(\nabla f)\|}{\trace(T)}\right)\|H_T-T(\nabla f)\|\right]\mu_f.
\eeQ
Moreover, from the assumption that $\Omega$ is contained in the ball of 
radius $B(p,R)$, we get that $\sd(r)\leqslant\sd(R)$ and thus
\[
\sigma_1\int_M\sd^2(r)\widetilde{\mu}_f 
\leqslant
\sd^2(R)V_f(\Omega)\sup_{\Omega}\left[\delta\trace(T)+
\sup_{\Omega}\left( \frac{\|H_T-T(\nabla f)\|}{\trace(T)}\right)
 t\|H_T-T(\nabla f)\|\right].
\]
Finally, by Lemma \ref{lem2}, we have
$$\dfrac{\displaystyle\int_M\sd(r)^2\widetilde{\mu}_{f}}{V_f(M)}\geqslant\dfrac{1}{\delta+\dfrac{\sup_M\|H_S-S(\widetilde{\nabla} f)\|^2}{\inf_M(\trace(S))^2}},$$
which give the desired estimate when $\delta\leqslant 0$:

\begin{eqnarray*}
\sigma_{1}\leqslant \sup_{\Omega}\left[ \delta\trace(T)+\sup_{\Omega}\left( \frac{\|H_T-T(\nabla f)\|}{\trace(T)}\right)\|H_T-T(\overline{\nabla} f)\|\right]\\ \times\left[\delta+\dfrac{\sup_M\|H_S-S(\widetilde{\nabla} f)\|^2}{\inf_M(\trace(S))^2}\right]\frac{V_f(\Omega)}{V_f(M)}\sd^2(R).
\end{eqnarray*}
{\it Case $\delta>0$.}
Like in the case $\delta\leqslant 0$, we have
\beqt\label{majsigma13}
\sigma_1\int_{M}\sd^2(r)\widetilde{\mu}_f\leqslant\int_{\Omega}\left(\trace(T)-\delta\left \langle TX^{\top},X^{\top}\right\rangle\right)\mu_f
\eeqt
by using $\frac{\sd(r)}{r}x_i$, $1\leqslant i\leqslant N$ as test functions. 
In addition, we also use another test function, $\cd(r)-\widetilde{\cd}$,
with $\widetilde{\cd}=\dfrac{1}{V_f(M)}\int_M\cd(r)\widetilde{\mu}_f$. By a computation analogue to the proof of Theorem \ref{thm1}, we have
\beQ
\sigma_1\int_M\left(\cd(r)-\widetilde{\cd}\right)^2\widetilde{\mu}_f&\leqslant&\int_{\Omega}\left\langle T\nabla \left(\cd(r)-\widetilde{\cd}\right), \nabla\left(\cd(r)-\widetilde{\cd}\right)\right\rangle\mu_f\\
&\leqslant&\int_{\Omega}\left\langle T\nabla\cd(r), \nabla\cd(r)\right\rangle\overline{\mu}_f\\
&\leqslant&\delta^2\int_{\Omega}\sd^2(r)\left\langle T\nabla r,\nabla r\right\rangle\mu_f\\
&\leqslant&\delta^2\int_{\Omega}\left\langle TX^{\top},X^{\top}\right\rangle\mu_f.
\eeQ
From this, we deduce 
\beqt\label{majsigma14}
\sigma_1\int_M\cd^2(r)\widetilde{\mu}_f\leqslant\delta^2\int_{\Omega}\left\langle TX^{\top},X^{\top}\right\rangle\mu_f+\frac{\sigma_1}{V_f(M)}\left(\int_M\cd(r)\widetilde{\mu}_f\right)^2.
\eeqt
Now, using the fact that $\cd^2+\delta\sd^2=1$, \eqref{majsigma14} plus $\delta$ times \eqref{majsigma13} gives
$$ \sigma_1V_f(M)\leqslant \delta\int_{\Omega}\trace(T)\mu_f+\left(\dfrac{\displaystyle\int_M\cd(r)\widetilde{\mu}_f}{V_f(M)}\right)^2$$
and so
$$\sigma_1V_f(M)\left( 1-\left(\frac{\int_M\cd(r)\widetilde{\mu}_f}{V_f(M)}\right)^2\right)\leqslant\delta\int_{\Omega}\trace(T)\mu_f.$$
Finally, since we have
\[
1-\left(\frac {\displaystyle\int_M\cd(r)\widetilde{\mu}_{f}}{V_f(M)}\right)^2\geqslant\frac{1}{1+\dfrac{\displaystyle\int_M\|H_S-S(\widetilde{\nabla} f)\|^2\widetilde{\mu}_{f}}{\delta \inf(\trace(S))^2V_f(M)}},
\]
by Lemma \ref{lem2}, we get
\[
\sigma_{1}\leqslant \frac{\displaystyle \int_{\Omega}\trace(T)\mu_{f}}{V_f(M)}\left(\delta  +\frac{ \displaystyle\int_{M}\|H_S-S(\widetilde{\nabla} f)\|^2\widetilde{\mu}_f}{V_f(M)\inf\left(\trace(S)\right)^2}\right),
\]
and this concludes the proof.
\end{proof}

\begin{proof}[Proof of Theorem \ref{thm3}]
Here again we consider differently the two cases. \\{\it Case $\delta\leqslant0$.}
First, we recall the variational characterization of $\alpha_1$ (see \cite{DKL})
 \begin{equation}\label{characalpha1}
 \alpha_1=\inf\left\{\dfrac{\displaystyle\int_{\Omega}\|\overline{\nabla} u\|^2dv_{\overline{g}}+b\int_{M}\|\nabla u\|^2dv_{\widetilde{g}}}{\displaystyle\int_{M}u^2dv_{\widetilde{g}}}\ \Bigg|\int_{M}udv_{\widetilde{g}}=0\right\}.
 \end{equation}
As in the proof of Theorem \ref{thm2}, we use $\dfrac{\sd(r)}{r}x_i$ as test functions, where $r$ is the geodesic distance to the center of mass $p$ of $\Omega$ and $\{x_1,\cdots,x_N\}$ the normal coordinates of $\overline{M}$ centered at $p$. Hence, we have
\begin{eqnarray*}
\alpha_1\int_M\sum_{i=1}^N\frac{\sd^2(r)}{r^2}x_i^2dv_{\widetilde{g}}
\leqslant
\int_{\Omega}\sum_{i=1}^N\left\|\overline{\nabla}\left(\frac{\sd(r)}{r}x_i\right)\right\|dv_{\overline{g}}+b\int_{M}\sum_{i=1}^N\left\|\nabla\left(\frac{\sd(r)}{r}x_i\right)\right\|dv_{\widetilde{g}}
\end{eqnarray*}
that is, 
\begin{eqnarray} \label{quotientalpha1}
\begin{aligned}
\alpha_1\int_M\sd(r)^2dv_{\widetilde{g}} & \leqslant 
\int_{\Omega} \left(n-\delta\|X^{\top}\|^2\right)dv_{\overline{g}} \\
&+ b\int_{M} \left((n-1)-\delta\|X^{\widetilde{\top}}\|^2\right)dv_{\widetilde{g}},
\end{aligned}
\end{eqnarray}
where we use use inequality \eqref{ineqgrosjean} twice, once on $\Omega$ 
for the first term and once on $M$ for the second term. Moreover we have, 
$\|X^{\widetilde{\top}}\|\leqslant\|X^{\top}\|\leqslant \|X\|=\sd(r)$ and since 
$\delta$ is nonpositive, $\sd$ is an increasing function. So, by the 
assumption that $\Omega$ is contained in the ball $B(p,R)$, we have
\begin{eqnarray} \label{quotientalpha2}
\begin{aligned}
\alpha_1\int_M\sd^2(r)dv_{\widetilde{g}} &\leqslant
\left(n-\delta\sd(R)^2\right)V(\Omega)+b\left(n-1)-\delta\sd(R)\right)V(M) \\
&\leqslant \Big(nV(\Omega)+b(n-1)V(M)\Big)
-\delta\sd(R)^2\Big(V(\Omega)+bV(M)\Big).
\end{aligned}
\end{eqnarray}
We conclude by applying Lemma \ref{lem2}, which says for $f=0$,
$$\dfrac{\displaystyle\int_M\sd(r)^2dv_{\widetilde{g}}}{V(M)}\geqslant\dfrac{1}{\delta+\dfrac{\sup_M\|H_S\|}{\inf_M(\trace(S))^2}},$$
to obtain the desired estimate when $\delta\leqslant 0$, that is,
$$\alpha_1\leqslant \Big[n\frac{V(\Omega)}{V(M)}+b(n-1)-\delta\sd^2(R)\left(\frac{V(\Omega)}{V(M)}+b\right)\Big]\left(\delta+\dfrac{\sup_M\|H_S\|^2}{\inf_M(\trace(S))^2}\right).$$
{\it Case $\delta>0$.} Like in the cas $\delta\leqslant 0$, using the functions $\dfrac{\sd(r)}{r}x_i$ as test functions in the variational characterization of $\alpha_1$, we get
\begin{equation}\label{alphasd}
\alpha_1\int_M\sd^2(r)dv_{\widetilde{g}}\leqslant\int_{\Omega} \left(n-\delta\|X^{\top}\|^2\right)dv_{\overline{g}}+b\int_{M} \left((n-1)-\delta\|X^{\widetilde{\top}}\|^2\right)dv_{\widetilde{g}}.
\end{equation}
Moreover, we use another test function, $\cd(r)-\widetilde{\cd}$ with 
$\widetilde{\cd}=\dfrac{1}{V(M)}\int_M\cd(r)dv_{\widetilde{g}}$.
\beQ
\alpha_1\int_M\left(\cd(r)-\widetilde{\cd}\right)^2dv_{\widetilde{g}} &\leqslant&\int_{\Omega}\left\langle \nabla \left(\cd(r)-\widetilde{\cd}\right), \nabla\left(\cd(r)-\widetilde{\cd}\right)\right\rangle dv_g \\
&&+ b\int_M\left\langle \widetilde{\nabla} \left(\cd(r)-\widetilde{\cd}\right), \widetilde{\nabla}\left(\cd(r)-\widetilde{\cd}\right)\right\rangle dv_{\widetilde{g}}\\
&\leqslant&\int_{\Omega}\left\langle \nabla\cd(r), \nabla\cd(r)\right\rangle dv_g+b\int_M\left\langle \widetilde{\nabla} \cd(r), \widetilde{\nabla}\cd(r)\right\rangle dv_{\widetilde{g}}\\
&\leqslant&\delta^2\int_{\Omega}\sd^2(r)\left\langle \nabla r,\nabla r\right\rangle dv_g+b\delta^2\int_M\sd^2(r)\left\langle \widetilde{\nabla} r,\widetilde{\nabla} r\right\rangle dv_{\widetilde{g}}\\
&\leqslant&\delta^2\int_{\Omega}\|X^{\top}\|^2 dv_g+b\delta^2\int_M\|X^{\widetilde{\top}}\|^2 dv_{\widetilde{g}},
\eeQ
where $X^{\top}=\sd(r)\nabla r$ is the tangent component of $X$ to
$\Omega$ and $X^{\widetilde{\top}}=\sd(r)\widetilde{\nabla} r$ is the 
part of $X$ tangent to $M$. From this, we get
\begin{equation} \label{alphacd}
\begin{aligned}
\alpha_1\int_M\cd(r)^2dv_{\widetilde{g}} \leqslant& 
\delta^2\int_{\Omega}\|X^{\top}\|^2 dv_g+b\delta^2
\int_M\|X^{\widetilde{\top}}\|^2 dv_{\widetilde{g}} \\
+& \dfrac{\alpha_1}{V(M)}\left(\int_M\cd(r)dv_{\widetilde{g}} \right)^2.
\end{aligned}
\end{equation}
Hence summing \eqref{alphacd} and $\delta$ times \eqref{alphasd}, using the fact that $\cd^2+\delta\sd^2=1$, gives
$$\alpha_1V(M)\leqslant \delta\Big( nV(\Omega)+b(n-1)V(M) \Big)+\dfrac{\alpha_1}{V(M)}\left(\int_M\cd(r)dv_{\widetilde{g}} \right)^2,$$
and so
$$\alpha_1V(M)\left(1-\left(\dfrac{1}{V(M)}\int_M\cd(r)dv_{\widetilde{g}} \right)^2\right)\leqslant \delta\Big( nV(\Omega)+b(n-1)V(M) \Big).$$
Moreover, from we have Lemma \ref{lem2} (with $f=0$), we have
$$1-\left(\frac {\displaystyle\int_M\cd(r)dv_{\widetilde{g}}}{V(M)}\right)^2\geqslant\frac{1}{1+\dfrac{\displaystyle\int_M\|H_S\|^2dv_{\widetilde{g}}}{\delta \inf(\trace(S))^2V(M)}}$$
which gives
$$\alpha_{1}V(M)\leqslant \delta\Big( nV(\Omega)+b(n-1)V(M) \Big)\left(\delta  +\frac{ \displaystyle\int_{M}\|H_S\|^2dv_{\widetilde{g}}}{V(M)\inf\left(\trace(S)\right)^2}\right)$$
and finally the desired estimate
\[
\alpha_1\leqslant \left( n\frac{V(\Omega)}{V(M)}+b(n-1)\right)
\left(\delta  +\frac{ \displaystyle\int_{M}\|H_S\|^2dv_{\widetilde{g}}}{V(M)
\inf\left(\trace(S)^2\right)}\right).
\]
This concludes the proof.
\end{proof}

\section{A remark about the case $\delta>0$}

The aim of this section is to compare the estimates in both case $\delta>0$
and $\delta\leqslant0$. Indeed, in the estimates of Theorems \ref{thm2}, the 
radius $R$ of a ball containing $\Omega$ appears when $\delta\leqslant0$
but not for the estimates for $\delta>0$. It turns out that when $\delta>0$, 
we can bound the radius $R$ from above in terms of $H_T$ and $\trace{T}$
and so obtain upper bounds comparable to those obtained for 
$\delta\leqslant0$. First, we have the following

\begin{prop}
Let $(\bar{M}^{n+p},\bar{g},\bar{\mu}_{f})$ be a weighted Riemannian 
manifold with sectional curvature $sect_{\bar{M}}\leqslant\delta$, with 
$\delta>0$, and $\bar{\mu}_{f}=e^{-f}dv_{\bar{g}}$. Let $(\Omega,g)$ 
be a compact Riemannian manifold with non-empty boundary $M$ 
isometrically immersed into $(\bar{M}^{n+p},\bar{g})$ by $X$. We 
endow $\Omega$ with the weighted measure $\mu_{f}=e^{-f}dv_g$.
Let $T$ be a symmetric, divergence-free and positive definite 
$(1,1)$-tensor on $\Omega$. If $\Omega$ is contained a ball of 
radius $R$, then 
\[
\sd^2(R)\left(\dfrac{\|H_T-T\nabla f\|_{\infty}^2}{\displaystyle\inf_{\Omega}(\trace(T))^2}+\delta \right)\geqslant 1.
\]
\end{prop}
\begin{proof}
From the second part of Lemma \ref{lem1} applied on $\Omega$, we have
$$\int_{\Omega}\cd(r)\trace(T)\mu_f\leqslant-\int_{\Omega}\langle X,H_T-T\nabla f\rangle\mu_f.$$
We recall that $X=\sd(r)\nabla r$, which gives
$$\int_{\Omega}\cd(r)\trace(T)\mu_f\leqslant\int_{\Omega}\|H_T-T\nabla f\|\sd(r)\mu_f.$$
We are in the case where $\delta>0$, so $\cd$ and $\sd$ are respectively decreasing and increasing on $[0,\frac{\pi}{2\sqrt{\delta}}]$. Moreover, $\Omega$ is contained in a ball of radius $R<\frac{\pi}{4\sqrt{\delta}}$ which implies that $\Omega$ is contained in a ball of radius $\frac{\pi}{2\sqrt{\delta}}$ centered at $p$. Hence, $r<R$ and so $cd(r)\geqslant \cd(R)$ and $\sd(r)\leqslant\sd(R)$ on $\Omega$. Thus, we get
$$\cd(R)\inf_{\Omega}(\trace(T))\leqslant \sd(R)\sup_{\Omega}\|H_T-T\nabla f\|.$$
We deduce easily from this and the fact that $\cd^2+\delta\sd^2=1$ that
$$\sd^2(R)\left(\dfrac{\|H_T-T\nabla f\|_{\infty}^2}{\displaystyle\inf_{\Omega}(\trace(T))^2}+\delta \right)\geqslant 1,$$
which concludes the proof of the proposition.
\end{proof}

Now, using the above proposition together with the estimate of Theorem \ref{thm2} in the case $\delta>0$, we get the following estimate
\beQ\sigma_{1}&\leqslant& \left(\dfrac{\|H_T-T\nabla f\|_{\infty}^2}{\displaystyle\inf_{\Omega}(\trace(T))^2}+\delta \right)\left(\delta  +\frac{ \displaystyle\int_{M}\|H_S-S(\nabla f)\|^2\widetilde{\mu}_f}{V_f(M)\inf\left(\trace(S)^2\right)}\right)\frac{\displaystyle \int_{\Omega}\trace(T)\mu_{f}}{V_f(M)}\sd^2(R)\\
&\leqslant&\displaystyle\sup_{\Omega}(\trace(T))\left(\dfrac{\|H_T-T\nabla f\|_{\infty}^2}{\displaystyle\inf_{\Omega}(\trace(T))^2}+\delta \right)\left(\delta  +\frac{ \displaystyle\int_{M}\|H_S-S(\nabla f)\|^2\widetilde{\mu}_f}{V_f(M)\inf\left(\trace(S)^2\right)}\right)\frac{\displaystyle V_f(\Omega)}{V_f(M)}\sd^2(R)
\eeQ
which is comparable to the estimate for the case $\delta\leqslant 0$:

\begin{eqnarray*}
\sigma_{1}\leqslant \sup_{\Omega}\left[ \delta\trace(T)+\sup_{\Omega}\left( \frac{\|H_T-T(\overline{\nabla} f)\|}{\trace(T)}\right)\|H_T-T(\overline{\nabla} f)\|\right]\\
\times\left[\delta+\dfrac{\sup_M\|H_S-S(\nabla f)\|^2}{\inf_M(\trace(S))^2}\right]\frac{V_f(\Omega)}{V_f(M)}\sd^2(R).
\end{eqnarray*}
\section{Data Availability Statement}
Data sharing not applicable to this article as no datasets were generated or analysed during the current study.

\end{document}